\documentclass[reqno,11pt]{amsart}
\usepackage{amssymb}
\usepackage{verbatim}
\newif\ifpdf
\ifpdf
  \usepackage[pdftex]{graphicx}
  \usepackage[pdftex]{hyperref}
\else
  \usepackage{graphicx}
\fi
\numberwithin{equation}{section}       
 \theoremstyle{plain}    
 \newtheorem{thm}{Theorem}[section]
 \numberwithin{equation}{section} 
 \numberwithin{figure}{section} 
 \theoremstyle{plain}
 \theoremstyle{plain}    
 \theoremstyle{plain}    
 \theoremstyle{plain}    
 \newtheorem{lem}[thm]{Lemma} 
 \theoremstyle{remark}
 \newtheorem{rem}[thm]{Remark}
 \theoremstyle{definition}

\theoremstyle{definition}

\newtheorem*{ackn}{Acknowledgement}

\newcommand{\C}{{\mathbf{C}}}

\newcommand{\N}{{\mathbf{N}}}

\newcommand{\R}{{\mathbf{R}}}

\newcommand{\PP}{{\mathbf{P}}}

\newcommand{\cE}{{\mathcal{E}}}
\newcommand{\cF}{{\mathcal{F}}}

\renewcommand{\a}{\alpha}

\newcommand{\MA}{\mathrm{MA}\,}

\newcommand{\vol}{\operatorname{vol}}

\begin{document}

\setcounter{tocdepth}{1}

\title{Equidistribution of Fekete points on complex manifolds}

\date{\today{}}

\author{Robert Berman, S\'{e}bastien Boucksom}

\address{Universit\'{e} Grenoble I\\
 Institut Fourier\\
 Saint-Martin d'Heres\\
 France}

\email{robertb@math.chalmers.se}

\address{CNRS-Universit\'{e} Paris 7\\
 Institut de Math\'{e}matiques\\
 F-75251 Paris Cedex 05\\
 France}

\email{boucksom@math.jussieu.fr}

\begin{abstract}
We prove the several variable version of a classical equidistribution
theorem for Fekete points of a compact subset of the complex plane, which settles a well-known conjecture in pluri-potential theory. The result is obtained as a special case of a
general equidistribution theorem for Fekete points in
the setting of a given holomorphic line bundle over a compact complex
manifold. The proof builds on our recent work {}``Capacities and
weighted volumes for line bundles''.
\end{abstract}
\maketitle

\section{Introduction}

A classical potential theoretic invariant of a compact set $E$ in
the complex plane $\C$ is given by its \emph{transfinite diameter:}\begin{equation}
d_{\infty}(E):=\lim_{k\rightarrow\infty}\left(\sup_{z_{0},...,z_{k}\in E}\prod_{0\leq i<j\leq k}\left|z_{i}-z_{j}\right|\right)^{2/k(k-1)}.\label{eq:intro
  diam}\end{equation}
 i.e.~the asymptotic geometric mean distance of points in $E$. A
configuration $z^{(k)}\in E^{N_{k}}$ of points achieving the
supremum for a fixed $k$ is called a $k$-\emph{Fekete configuration}
for $E$ (in classical terminology the corresponding points $(z_{i}^{(k)}),$
for $k$ fixed, are called \emph{Fekete points}). A basic classical
theorem (see \cite{s-t} for a modern reference and \cite{dei} for
the relation to Hermitian random matrices) asserts that Fekete configurations equidistribute on the \emph{equilibrium measure} $\mu_E$ of $E$, i.e.  
$$\frac{1}{k+1}\sum_{i=1}^{N_k}\delta_{z^{(k)}_{i}}\to\mu_E$$
as $k\to\infty$. 
In the logarithmic pluripotential
theory in $\C^{n}$ a higher dimensional version of the transfinite
diameter was introduced by Leja in 1959 
\begin{equation}
d_{\infty}(E):=\limsup_{k\to\infty}\left(\sup_{z_{1},\dots,z_{N_{k}}\in E}\left|\Delta(z_{1},...,z_{N_{k}})\right|\right)^{(n+1)!/nk^{n+1}},\label{eq:def of classical transf}\end{equation}
 where $\Delta(z_{1},...,z_{N_{k}})$ is the following higher-dimensional
Vandermonde determinant: \[
\Delta(z_{1},...,z_{N_{k}}):=\det(e^{\a}(z_{j}))_{|\a|\leq k,1\leq j\leq N_{k}},\]
 where $e^{\a}(x)=x_1^{\a_1}...x_n^{\a_n}$, $\a\in\N^n$ denote the monomials and $N_{k}$
is the number of monomials of degree at most $k$ (so that $N_{k}=k^{n}/n!+O(k^{n-1}))$.
It was shown by Zaharjuta \cite{za} that the limsup in formula (\ref{eq:def of classical transf})
is actually a true limit, thereby answering a conjecture by Leja.
However, the corresponding convergence of higher-dimensional {}``Fekete
configurations'' (towards the pluripotential theoretic equilibrium measure
of $E)$ has remained an open problem (see the survey \cite{l} by
Levenberg on approximation theory in $\C^{n},$ where it is pointed
out (p. 120) that {}``to this date, \emph{nothing} is known if $n>1$'').

There is also a weighted version of the previous setting where the
set $E$ is replaced by the \emph{weighted set} $(E,\phi),$ with
$\phi$ is a continuous function on the compact set $E$ (see the
appendix by Bloom in \cite{s-t}). The weighted transfinite diameter
$d_{\infty}(E,\phi)$ is then obtained by replacing $\left|\Delta(z_{1},...,z_{N_{k}})\right|$
in (\ref{eq:def of classical transf}) by its weighted counterpart
$$\left|\Delta(z_{1},...,z_{N_{k}})\right|e^{-k(\phi(z_{1})+...+\phi(z_{N}))}.$$

Even more generally, there is a variant of this weighted setting where
$E$ may be assumed to be an unbounded set in $\C^{n}$ provided
that $\phi$ growths sufficiently fast at infinity (cf.~the appendix
by Bloom in \cite{s-t}).

The aim of this note is to prove a generalized version of the conjecture
referred to above in the more general setting of a big line
bundle $L$ over a complex manifold $X$ (recall that from an analytic perspective $L$ is big iff it admits a singular Hermitian metric with strictly positive curvature current). In this setting the
role of $e^{-\phi}$ is played by a Hermitian metric on $L$ (we will
call the additive object $\phi$ a \emph{weight} - see \cite{b-b}
for further notation). A weighted subset $(E,\phi)$ will thus consist
of a subset $E$ of $X$ and a continuous weight on $L$. The {}``classical
setting'' above corresponds to the hyperplane line bundle $\mathcal{O}(1)$
over the complex projective space $\PP^{n}$ with $E$ a compact set
in the affine piece $\C^{n}$ and the space $H^{0}(kL)$ of global
holomorphic sections with values in $kL$ (the $k$-th tensor power
of $L,$ written in additive notation) may in this case be identified
with the space of all polynomials on $\C^{n}$ of total degree at
most $k.$ 

In order to state the theorem we first recall some further notation,
mostly taken from \cite{b-b}. Fix a compact subset $E$ of $X$ which
is not locally pluripolar and a continuous weight $\phi$ on the line
bundle $L\rightarrow X$. The \emph{equilibrium weight} of $(E,\phi)$ is defined by 
\begin{equation}
\phi_{E}=\sup\left\{ \varphi,\,\varphi\,\textrm{psh weight on$\, L$},\,\varphi\leq\phi\,\,\textrm{on$\, E$}\right\} ,\label{eq:extem metric}\end{equation}
where a weight $\varphi$ on $L$ is called \emph{psh} if its curvature $dd^c\varphi$ is a positive current. In the $\C^{n}$-case above,
$\phi_{E}$ is usually referred to as the \emph{weighted Siciak extremal
function} of $(E,\phi)$ (cf.~e.g.~the appendix
by Bloom in \cite{s-t}).
The \emph{equilibrium measure} of the weighted
set $(E,\phi)$ is then defined as the Monge-Amp\`ere measure \begin{equation}
\mu_{(E,\phi)}:=\MA(\phi_{E}^{*}),\label{eq:def of equil}\end{equation}
of the psh weight with minimal singularities $\phi_E^*$ (the usc envelope of $\phi_E$), that is the trivial extension to $X$ of the Bedford-Taylor wedge-product $(dd^c\phi_E^*)^n$ computed on a Zariski open subset where $\phi_E^*$ is locally bounded (see \cite{b-b}). 

Next let $(s_1,...,s_N)$ be a basis of $H^{0}(L)$ and consider the following
Vandermonde-type determinant 
$$\det(s)(x_{1},\dots,x_{N}):=\det(s_{i}(x_{j}))_{i,j}$$
 which is a holomorphic section of the pulled-back line bundle $L^{\boxtimes N}$
over the $N$-fold product $X^{N}$. 

A configuration of points $P=(x_{1},...,x_{N})\in E^N$
is called a \emph{Fekete configuration} for the weighted subset $(E,\phi)$ if it realizes the supremum on $E^N$ of the pointwise length $\left|\det(s)\right|_\phi$ of $\det(s)$ with respect to the metric induced by $\phi$. It is a basic observation that the definition of a Fekete configuration
is actually independent of the choice of the basis $(s_{i})$. Indeed, since
the top exterior power of the vector space $H^{0}(L)$ is one-dimensional,
replacing $s_{i}$ by $s_{i}'$ gives $\det(s')=c\det(s'),$ where
$c$ is a non-zero complex number. 

As a last piece of notation, given a configuration $P=(x_1,...,x_N)\in X^N$ for some $N$, we will denote by $$\delta_P:=\frac{1}{N}\sum_{i=1}^{N}\delta{x_i}$$
the associated probability measure.

\begin{thm}
\label{thm:fekete conv} Let $L\rightarrow X$ be a big line bundle
over a compact complex manifold. Assume that $P_k\in X^{N_k}$
is a Fekete configuration for $(E,k\phi)$ for each $k$ large enough. Then this sequence of configurations equidistributes towards the equlibrium measure of $(E,\phi)$, i.e.
$$\lim_{k\to\infty}\delta_{P_k}=M^{-1}\mu_{(E,\phi)}$$
weakly as measures on $X$, where $M$ is the total mass of $\mu_{(E,\phi)}$. 
\end{thm}
Here $N_k:=h^0(kL)\simeq\vol(L)k^n/n!$, where $\vol(L)>0$ is the \emph{volume} of the big line bundle $L$, which is also equal to the total mass $M$ of $\mu_{(E,\phi)}$ (see~\cite{b-b}). 

In the one dimensional case, i.e.~when $X$ is a complex curve and
if $(E,\phi)$ is taken as $(E,0)$ where $E$ is a compact set contained
in an affine piece of $X$ where $L$ has been trivialized, the theorem
was obtained, using completely different methods, by Bloom and Levenberg
\cite{b-l1} (see also \cite{g-m-s} for the case when $X$ has genus
zero and for a discussion of related interpolation problems in $\R^{n}.$)

The proof of Theorem \ref{thm:fekete conv} builds on our previous
work \cite{b-b}, where it was shown that (minus the logarithm of)
the transfinite diameter of $(E,\phi),$ considered as a functional
on the affine space of all continuous weights on $L$ (for $E$ fixed)
is Fr\'echet differentiable: its differential at the weight $\phi$
may be represented by the corresponding equilibrium measure $\mu_{(E,\phi)}.$
A similar argument was used very recently by the first author and
Witt Nyström in \cite{b-w} to obtain very general convergence results
for {}``Bergman measures'' (corresponding to Christoffel-Darboux
functions in classical terminology). In fact, a unified treatment
of these two convergence results can be given, which also includes
the equidistribution of generic points of {}``small height'' on
an arithmetic variety obtained in \cite{b-b} (which in turn generalizes
Yuan's arithmetic equidistribution theorem \cite{yu}). This will
be further investigated elsewhere.

\begin{rem}
In the case that $L$ is ample the differentiability property referred
to above can also be deduced from the Bergman kernel asymptotics for
smooth weights in \cite{berm2} as explained in section 1.4 in \cite{b-b}
(see also remark 11.2 in \cite{b-b} ). For an essentially elementary
proof of these asymptotics in the weighted case in $\C^{n}$ see \cite{berm1}.
\end{rem}
\begin{ackn}
It is a pleasure to thank Bo Berndtsson, Jean-Pierre Demailly and
David Witt Nyström for stimulating discussions related to the topic
of this note. We are also grateful to Norman Levenberg for his interest
in and comments on the paper \cite{b-b}.
\end{ackn}

\subsection{Proof of Theorem \ref{thm:fekete conv}}
Given a basis $s=(s_1,...,s_N)$ of $H^0(L)$, set 
$$D_\phi:=\log|\det(s)|_\phi,$$ 
so that $D_\phi$ achieves its maximum on $E^N$ exactly at Fekete configurations of $(E,\phi)$ by definition. As noticed above, $D_\phi$ only depends on the choice of the basis $s$ up to an additive constant. If $P=(x_1,...,x_N)\in X^N$ is a given configuration, the more explicit formula
$$D_\phi(P)=\log\left|\det(s_i(x_j)\right|-\left(\phi(x_1)+...+\phi(x_N)\right)$$
clearly shows that $\phi\mapsto D_\phi(P)$ is an affine function, with linear part given by integration against $-N\delta_P$. 

In order to normalize $D_\phi$, we fix an auxiliary (smooth positive) volume form $\mu$ on $X$ and smooth metric $e^{-\psi}$ on $L$ and take $(s_1,...,s_N)$ to be an orthonormal basis of $H^0(L)$ with respect to the corresponding $L^2$ scalar product. It is easily seen (cf.~\cite{b-b}) that $D_\phi$ becomes independent of the choice of such an orthonormal basis. The main result of~\cite{b-b} implies that 
\begin{equation}\label{equ:limit}\lim_{k\to\infty}\frac{(n+1)!}{k^{n+1}}\sup_{E^{N_k}} D_{k\phi}=\cE(\psi_X,\phi_E^*)
\end{equation}
where $\psi_X$ is the equilibrium weight of $(X,\psi)$ and $\cE$ denotes the Aubin-Yau energy, whose precise formula doesn't matter here. The main point for what follows is that 
$$\phi\mapsto\cE(\psi_X,\phi_E^*)$$
is Fr\'echet differentiable on the space of all continuous weights on $L$, with derivative at $\phi$ in the tangent direction $v\in C^0(X)$ given by
$$(n+1)\int_X v\,\mu_{(E,\phi)}.$$
This is indeed the content of Theorem 5.7 of \cite{b-b}. Now let $P_k\in E^{N_k}$ be a sequence of configurations, and set 
$$F_k(\phi):=-\frac{1}{kN_k} D_{k\phi}(P_k)$$
and 
$$G(\phi):=\frac{1}{(n+1)M}\cE(\phi_E^*,\psi_X).$$
We thus see that 
$$\liminf_{k\to\infty}F_k(\phi)\ge G(\phi),$$
and furthermore
$$\lim_{k\to\infty}F_k(\phi)=G(\phi)$$
if $P_k\in E^{N_k}$ is a Fekete configuration for $(E,k\phi)$, as follows from (\ref{equ:limit}) and the fact that $N_k\simeq M k^n/n!$. 

As noticed above, the functional $F_k$ is affine, with linear part given by integration against $\delta_{P_k}$. On the other hand the differentiability property of the energy writes
$$\frac{d}{dt}_{t=0}G(\phi+t v)=M^{-1}\int_X v\,\mu_{(E,\phi)}.$$

The proof of Theorem~\ref{thm:fekete conv} is thus concluded by the following elementary result applied to $f_k(t):=\cF_k(\phi+tv)$ and $g(t):=G(\phi+tv)$, taking $P_k\in E^{N_k}$ to be a Fekete configuration for $(E,k\phi)$ as in the Theorem.

\begin{lem}
Let $f_k$ by a sequence of concave functions
on $\R$ and let $g$ be a function on $\R$ such that 
\begin{itemize}
\item  $\liminf_{k\to\infty} f_k\ge g$.
\item $\lim_{k\to\infty} f_k(0)=g(0)$.
\end{itemize}
If the $f_k$ and $g$ are differentiable at $0$, then 
$$\lim_{k\to\infty}f_k'(0)=g'(0).$$
\end{lem}
\begin{proof}
Since $f_k$ is concave, we have 
$$f_k(0)+f_k'(0)t\ge f_k(t)$$ 
and it follows that
$$\liminf_{k\to\infty} t f_k'(0)\ge g(t)-g(0).$$
The result now follows by first letting $t>0$ and then $t<0$ tend to $0$. 
\end{proof}
The same lemma underlies the proof of Yuan's equidistribution theorem given in~\cite{b-b}. It is in fact inspired by the variational principle in the original proof by Szpiro-Ullmo-Zhang. The case of concave functions $f_k$ pertains to the situation considered in \cite{b-w}.

\end{document}